\documentclass[11pt]{amsart}

\usepackage{amsmath}
\usepackage{amsfonts}
\usepackage{amssymb}
\usepackage{amscd}
\usepackage{amsthm}
\usepackage{framed}
\usepackage{fullpage}
\usepackage{graphicx}
\usepackage{latexsym}
\usepackage[numbers]{natbib}  

\usepackage{hyperref}
\hypersetup{colorlinks,linkcolor={red},citecolor={blue},urlcolor={blue}}

\newtheorem{theorem}{Theorem}[section]
\newtheorem{lemma}[theorem]{Lemma}

\newtheorem{corollary}[theorem]{Corollary}

\theoremstyle{definition}

\numberwithin{equation}{section}

\newcommand{\CC}{\mathbb C}

\newcommand{\cD}{\mathcal D}
\newcommand{\cA}{\mathcal A}

\newcommand{\PP}{\mathbb P}

\newcommand{\ZZ}{\mathbb Z}

\newcommand{\Orth}{\mathop{\null\mathrm {O}}\nolimits}

\newcommand{\rank}{\mathop{\mathrm {rk}}\nolimits}
\newcommand{\latt}[1]{{\langle{#1}\rangle}}

\def\det{\operatorname{det}}

\begin{document}

\title[$2$-reflective lattices of signature $(n,2)$ with $n\geq 8$]{$2$-reflective lattices of signature $(n,2)$ with $n\geq 8$}

\author{Haowu Wang}

\address{Center for Geometry and Physics, Institute for Basic Science (IBS), Pohang 37673, Korea}

\email{haowu.wangmath@gmail.com}


\date{\today}


\begin{abstract}
An even lattice $M$ of signature $(n,2)$ is called $2$-reflective if there is a non-constant modular form for the orthogonal group of $M$ which vanishes only on quadratic divisors orthogonal to $2$-roots of $M$. In 2017 Ma \cite{Ma17} proved that there are only finitely many $2$-reflective lattices of signature $(n,2)$ with $n\geq 7$. In this paper we extend the finiteness result of Ma to $n\geq 5$ and show that there are exactly forty-two $2$-reflective lattices of signature $(n,2)$ with $n\geq 8$.
\end{abstract}

\maketitle

\section{Introduction}
Let $M$ be an even lattice of signature $(n,2)$ with $n\geq 3$.  The type IV Hermitian symmetric domain $\cD(M)$ attached to $M$ is a connected component of the space
\begin{equation*}
\{[\mathcal{Z}] \in  \PP(M\otimes \CC):  (\mathcal{Z}, \mathcal{Z})=0, (\mathcal{Z},\bar{\mathcal{Z}}) < 0\}.
\end{equation*}
We denote by $\Orth^+ (M)$ the orthogonal group preserving $\cD(M)$ and $M$. Let $\Gamma$ be a finite-index subgroup of $\Orth^+(M)$ and $k$ be an integer. A holomorphic function $F$ on the affine cone 
$$
\cA(M)=\{ \mathcal{Z} \in M\otimes\CC : [\mathcal{Z}] \in \cD(M) \}
$$
is called a \textit{modular form} of weight $k$ and character $\chi$ for $\Gamma$
if it satisfies
\begin{align*}
F(t\mathcal{Z})&=t^{-k}F(\mathcal{Z}), \quad \forall t \in \CC^\times,\\
F(g\mathcal{Z})&=\chi(g)F(\mathcal{Z}), \quad \forall g\in \Gamma.
\end{align*}
A non-constant modular form $F$ is called \textit{reflective} if it vanishes only on quadratic divisors  
$$
l^\perp = \{ [\mathcal{Z}] \in \cD(M) : (\mathcal{Z}, l)=0 \}
$$
for some roots $l\in M$, that is, $l$ are primitive positive-norm vectors of $M$ whose associated reflection 
$$
\sigma_l: x\mapsto x-\frac{2(l,x)}{(l,l)}l, \quad x\in M
$$
fixes the lattice $M$, i.e. $\sigma_l \in \Orth^+(M)$. Bruinier's result \cite{Bru02, Bru14} yields that reflective modular forms can usually be constructed as automorphic Borcherds products \cite{Bor95, Bor98}.

Reflective modular forms first appeared in the works of Borcherds \cite{Bor95, Bor98} and Gritsenko--Nikulin \cite{GN98a, GN98b}. They have many important applications to generalized Kac--Moody algebras \cite{Bor92, Bor98, GN98a, GN98b, GN02, Sch06}, hyperbolic reflection groups \cite{Bor00, GN98a}, birational geometry of moduli spaces \cite{Bor96, BKP98, GN96, GN00, GHS07, GH14, Gri18, Ma18} and the classification and construction of free algebras of modular forms \cite{Wan21}. It is a common belief that reflective modular forms are very rare. In 1998 Gritsenko and Nikulin \cite[Conjecture 2.2.1]{GN98a} proposed the arithmetic mirror symmetry conjecture, stating that the number of lattices with a reflective modular form is finite up to scaling. Since then, many classifications of reflective modular forms have been obtained \cite{GN02, Bar03, Sch06, Sch17, Ma17, Ma18, Dit18, Wan18, GN18, Wan19, Wan22}. 

In this paper we study $2$-reflective modular forms, the most basic class of reflective modular forms. A reflective modular form on $\Gamma < \Orth^+(M)$ is called \textit{$2$-reflective} if its zero divisor is a linear combination of quadratic divisors $l^\perp$ for $l\in M$ with $(l,l)=2$. An even lattice $M$ is called \textit{$2$-reflective} if there is a $2$-reflective modular form for some finite-index subgroup of $\Orth^+(M)$. It follows from the symmetrization trick that if $M$ is $2$-reflective then there is a $2$-reflective modular form for $\Orth^+(M)$. Gritsenko and Nikulin observed \cite{GN96, GN00} that $2$-reflective modular forms are related to $K3$ surfaces and Calabi--Yau manifolds, in particular, they have a geometric interpretation as the automorphic discriminant of the moduli space of lattice-polarized $K3$ surfaces. 

There is some relation between $2$-reflective modular forms and hyperbolic $2$-reflection groups. Given a hyperbolic even lattice $S$. Let $W$ be the subgroup of $\Orth^+(S)$ generated by reflections associated with $2$-roots and $\mathcal{M}$ be an associated fundamental polyhedron. If the subgroup $A(\mathcal{M})$ of $\Orth^+(S)$ fixing $\mathcal{M}$ has finite index in the quotient group $\Orth^+(S) / W$ then $S$ is called \textit{$2$-reflective}. A  $2$-reflective hyperbolic lattice $S$ is called \textit{elliptic} if $A(\mathcal{M})$ is finite, otherwise it is called \textit{parabolic}. Nikulin and Vinberg \cite{Nik80b, Nik81, Nik84, Vin84, Nik96, Vin07} proved that the set of $2$-reflective hyperbolic lattices $S$ with $\rank(S)\geq 3$ is finite and gave a full classification of elliptic $2$-reflective hyperbolic lattices (see e.g. \cite[Section 3.2]{GN18} for a list). This classification was motivated by one result of Pjatecki\u{\i}-\v{S}apiro and \v{S}afarevi\v{c} \cite{PS71}, proving that a complex algebraic $K3$ surface with the Picard lattice $S$ has finite automorphism group if and only if $S$ is elliptic $2$-reflective. 

The arithmetic mirror symmetry conjecture of Gritsenko--Nikulin \cite[Section 2]{GN00} predicts that 
\begin{itemize}
    \item[(a)] there are only finitely many $2$-reflective lattices of signature $(n,2)$ with $n\geq 3$;
    \item[(b)] if $M$ has a $2$-reflective modular form $F$ then the hyperbolic lattice $c_M^\perp / \ZZ c$ is $2$-reflective for any primitive norm zero vector $c\in M$ such that $F$ vanishes on some $v^\perp$ with $v\in c_M^\perp / \ZZ c$.
\end{itemize}
Part (b) was proved in 2003 by Looijenga \cite[Corollary 5.11]{Loo03}. Part (a) was later proved in 2017 by Ma \cite{Ma17} for $n\geq 7$. Part (b) of the Gritsenko--Nikulin conjecture does not lead to an exact classification of $2$-reflective lattices, because the classification of parabolic $2$-reflective hyperbolic lattices is unknown. Ma’s proof is in algebraic geometry and his result is ineffective to classify $2$-reflective lattices. It turns out that one may need new ways to attack this problem.

In \cite{Wan19} the author developed an approach based on the theory of Jacobi forms \cite{EZ85, Gri18} to classify $2$-reflective lattices. Let $U$ be an even unimodular lattice of signature $(1,1)$ and $L$ be an even positive definite lattice. This approach yields that  if $2U\oplus L$ has a $2$-reflective modular form then either $L$ has no $2$-roots or $L$ contains a sublattice of the same rank generated by $2$-roots satisfying some strong constrains. After detailed analysis, it was found that there are exactly fifty-one $2$-reflective lattices of type $2U\oplus L$, where $L$ has $2$-roots. 

In this paper we improve Ma's result by means of Part (b) of the Gritsenko--Nikulin conjecture.

\begin{theorem}\label{MTH0}
There are only finitely many $2$-reflective lattices of signature $(n,2)$ with $n\geq 5$.     
\end{theorem}

We also give a complete classification of $2$-reflective lattices of signature $(n,2)$ with $n\geq 8$. 

\begin{theorem}\label{MTH}
There are exactly forty-two $2$-reflective lattices of signature $(n,2)$ with $n\geq 8$ up to isomorphism. They are formulated in Table \ref{tab:2-reflective}. In particular, there is no $2$-reflective lattice of signature $(n,2)$ for $n\geq 13$ and $n\neq 18, 19, 26$. 
\end{theorem}

Theorem \ref{MTH} has been proved in \cite[Theorem 1.1]{Wan19} for $n\geq 14$. In this paper we give a simpler proof and extend it to $n\geq 13$ (see Theorem \ref{th:n>=13}). To prove Theorem \ref{MTH} for $n\leq 12$, we combine the main results of \cite{Wan19} and some techniques of lattices so that we can drop the $2U$ assumption. The proof does not rely on the classification of $2$-reflective hyperbolic lattices. Note that there are $10$ lattices in Table \ref{tab:2-reflective} that do not appear in the table of \cite[Theorem 1.2]{Wan19}.

We remark that the converse of Part (b) of the Gritsenko--Nikulin conjecture does not hold. For example, there are indeed elliptic $2$-reflective hyperbolic lattices $S$ of rank $13\leq \rank(S) \leq 17$, but there is no $2$-reflective lattice of signature $(\rank(S), 2)$. We also remark that there are $2$-reflective lattices in Table \ref{tab:2-reflective} which induces parabolic $2$-reflective hyperbolic lattices through Part (b) of the Gritsenko--Nikulin conjecture, such as $U\oplus E_8(2)$, $U(2)\oplus 8A_1$ and $U\oplus E_6'(3)$. 

We have mentioned that elliptic $2$-reflective hyperbolic lattices are related to $K3$ surfaces with finite automorphism group. It would be interesting to know if $2$-reflective lattices in Table \ref{tab:2-reflective} correspond to a certain more special class of $K3$ surfaces (see e.g. \cite[Section 3]{GN00}). 

This paper is organized as follows. In Section \ref{sec:lemmas} we prove some technical lemmas about lattices and $2$-reflective modular forms. Section \ref{sec:Ma} is devoted to the proof of Theorem \ref{MTH0}. In Section \ref{sec:proof} we prove Theorem \ref{MTH} and give three corollaries.  

\begin{table}[ht]
\caption{$2$-reflective lattices of signature $(n,2)$ with $n\geq 8$}
\label{tab:2-reflective}
\renewcommand\arraystretch{1.1}
\[
\begin{array}{|l|l|}
\hline
n & \text{$2$-reflective lattice} \\
\hline 
26 & 2U\oplus 3E_8 \\
\hline 
19 & 2U\oplus 2E_8\oplus A_1 \\
\hline 
18 & 2U\oplus 2E_8 \\
\hline 
12 & 2U\oplus E_8\oplus 2A_1 \\
\hline 
11 & 2U\oplus D_4\oplus 5A_1, \quad 2U\oplus 2D_4\oplus A_1, \quad 2U\oplus D_8\oplus A_1, \quad 2U\oplus E_8\oplus A_1 \\
\hline
10 & 2U\oplus E_8, \quad 2U\oplus D_8, \quad 2U\oplus 2D_4, \quad 2U\oplus D_8'(2), \\
&2U\oplus E_7\oplus A_1, \quad 2U\oplus D_6\oplus 2A_1, \quad 2U\oplus D_4\oplus 4A_1, \quad 2U\oplus 8A_1, \\
&2U\oplus E_8(2), \quad U\oplus U(2)\oplus E_8(2), \quad U\oplus U(2)\oplus 8A_1, \quad 2U(2)\oplus 8A_1\\
\hline 
9 & 2U\oplus D_7, \quad 2U\oplus A_7, \quad 2U\oplus E_7, \quad 2U\oplus E_6\oplus A_1, \quad 2U\oplus D_6\oplus A_1, \\ 
&2U\oplus D_4\oplus 3A_1, \quad 2U\oplus 7A_1, \quad U\oplus U(2)\oplus 7A_1, \quad 2U(2)\oplus 7A_1 \\
\hline 
8 & 2U\oplus D_6, \quad 2U\oplus A_6, \quad 2U\oplus 2A_3, \quad 2U\oplus 3A_2, \quad 2U\oplus E_6, \\
&2U\oplus D_5\oplus A_1, \quad 2U\oplus A_5\oplus A_1, \quad 2U\oplus D_4\oplus 2A_1, \quad 2U\oplus 6A_1, \\
&2U\oplus E_6'(3), \quad U\oplus U(3)\oplus E_6'(3), \quad U\oplus U(2)\oplus 6A_1, \quad 2U(2)\oplus 6A_1 \\
\hline
\end{array} 
\]
\end{table}

\section{Basic lemmas}\label{sec:lemmas}
In this section we collect and prove some basic lemmas about lattices and $2$-reflective modular forms that we will use later. 

Let $M$ be an even lattice of rank $\rank(M)$ with a bilinear form $(-,-)$ and dual lattice $M'$. Let $A_M=M'/M$ denote the discriminant group of $M$. We denote the minimal number of generators of $A_M$ by $l(M)$ and the maximal order of elements of $A_M$ by $e(M)$. The integers $l(M)$ and $e(M)$ are called the length and exponent of $A_M$, respectively. Let us fix a basis of the (unique) even unimodular lattice of signature $(1,1)$ as 
$$
U=\ZZ e + \ZZ f, \quad (e,e)=(f,f)=0, \quad (e,f)=1.
$$
For any positive integer $a$, we denote by $M(a)$ the lattice with abelian group $M$ and rescaled bilinear form $a(-,-)$. The level of $M$ is the smallest positive integer $m$ such that $m(x,x)\in 2\ZZ$ for all $x\in M'$. An embedding $M_1 \hookrightarrow M$ of even lattices is called primitive if $M/M_1$ is a free $\ZZ$-module. A given embedding $M \hookrightarrow M_1$ of even lattices, for which $M_1/M$ is a finite abelian group, is called an even overlattice of $M$. For any $v\in M$ we define an ideal of $\ZZ$ as
$$
(v,M):=\{ (v,x) : x\in M \}.
$$
We use $A_n$, $D_n$, $E_6$, $E_7$ and $E_8$ to denote the usual irreducible root lattices (see \cite{CS99}). We refer to \cite{CS99} for the notion of the genus of a lattice. 

\begin{lemma}\label{lem:U}
Let $M$ be an even lattice of signature $(n,1)$ with $n\geq 2$. If the length of $A_M$ satisfies that $l(M)\leq n-2$, then there exists an even positive definite lattice $L$ such that $M=U\oplus L$.
\end{lemma}
\begin{proof}
It is a direct consequence of Nikulin's results \cite{Nik80} (see e.g. \cite[Lemma 2.3]{Wan19} for a proof).    
\end{proof}

\begin{lemma}\label{lem:maximal}
Let $M$ be a maximal even lattice of signature $(n,2)$ with $n\geq 5$. Then $M$ can be represented as $M=2U\oplus L$ for some even positive definite lattice $L$.    
\end{lemma}
\begin{proof}
Let $c$ be a primitive norm zero vector of $M$. Since $M$ is maximal, $(c,M)=\ZZ$, which yields a decomposition $M=U\oplus K$. Since $K$ has signature $(n-1,1)$ and $\rank(K)=n\geq 5$, there is a primitive norm zero vector of $K$ denoted $c_1$. Similarly, $(c_1,K)=\ZZ$ and $K=U\oplus L$ for some $L$. We then obtain the desired decomposition. 
\end{proof}

\begin{lemma}\label{lem:exponent}
Let $M$ be an even lattice of signature $(n,2)$ with $n\geq 8$. There exists an even overlattice $M_1$ of $M$ satisfying the following conditions
\begin{enumerate}
    \item $M_1$ can be represented as $2U\oplus L$;
    \item $A_M$ and $A_{M_1}$ have the same exponent, i.e. $e(M)=e(M_1)$; 
    \item the length of $A_{M_1}$ satisfies that $l(M_1)\leq 5$. 
\end{enumerate}
\end{lemma}
\begin{proof}
This follows from \cite[Lemma 1.7]{Ma18} and its proof.     
\end{proof}

\begin{lemma}\label{lem:2-roots}
Let $L$ be an even positive definite lattice of rank $\rank(L)$. If the $2$-component of $A_L$ has length $l(A_L)_2 \leq \rank(L)-3$ and the $p$-component of $A_L$ has length $l(A_L)_p\leq \rank(L)-2$ for any odd prime $p$, then there is a class in the genus of $L$ which has $2$-roots.
\end{lemma}
\begin{proof}
Recall that $U=\ZZ e + \ZZ f$ with $e^2=f^2=0$ and $(e,f)=1$.
We define $M=U\oplus L$. Let us fix $v=e+f$ and $u=e-f$. Note that $v^2=2$ and $u^2=-2$. The orthogonal complement of $v$ in $M$ has the form $M_v=\ZZ u \oplus L$, so it has signature $(\rank(L),1)$. By assumptions, we have
\begin{align*}
&l(A_{M_v})_2 = l(A_L)_2 + 1 \leq \rank(L)-2,\\
&l(A_{M_v})_p = l(A_L)_p \leq \rank(L)-2, \quad \text{for any odd prime $p$}.
\end{align*}
Therefore, $l(M_v)\leq \rank(L)-2$. 
By Lemma \ref{lem:U}, there exists an even positive definite lattice $L_0$ such that $M_v=U\oplus L_0$. Since $U\oplus L_0\oplus \ZZ v$ has an even overlattice isomorphic to $M$, there exists an even overlattice $T$ of $L_0\oplus \ZZ v$ satisfying $M\cong U\oplus T$. By construction, $v\in T$, so $T$ has $2$-roots. Thus $T$ gives a desired class in the genus of $L$.
\end{proof}

We recall some basic properties of $2$-elementary lattices. An even lattice $M$ is called \textit{$2$-elementary} if $A_M\cong (\ZZ / 2\ZZ)^a$ for some non-negative integer $a$. The genus of a $2$-elementary lattice is described by Nikulin \cite[Theorem 3.6.2]{Nik80}. In particular, we have the following.

\begin{lemma}\label{lem:2-elementary}
Let $M$ be a $2$-elementary even lattice of signature $(n,2)$ with $n\geq 3$. Suppose that $A_M \cong (\ZZ / 2\ZZ)^a$ for some non-negative integer $a$.  Then the following holds.
\begin{enumerate}
    \item $a\leq n+2$ and $n+a$ is even.
    \item There are at most two distinct $M$ up to isomorphism when $n$ and $a$ are fixed. 
    \item When $n, a$ are fixed and $4$ does not divide $n-2$, $M$ is unique up to isomorphism if it exists. 
\end{enumerate}
\end{lemma}

We now give some lemmas about $2$-reflective modular forms and $2$-reflective lattices.

\begin{lemma}[Lemma 2.3 in \cite{Ma17}]\label{Lem:reductionMa}
If $M$ is $2$-reflective, then any even overlattice of $M$ is also $2$-reflective.  If $M$ is not $2$-reflective, neither is any finite-index sublattice
of $M$.
\end{lemma} 

\begin{lemma}[Lemma 5.2 in \cite{Wan19}]\label{lem:pullback}
Let $M$ be an even lattice of signature $(n,2)$ with $n\geq 3$ and $L$ be an even positive definite lattice. If $M\oplus L$ is $2$-reflective, then $M$ is also $2$-reflective.    
\end{lemma}

We now introduce a particular class of $2$-reflective modular forms. A modular form for $\Orth^+(M)$ is called \textit{complete $2$-reflective} if its zero divisor is a linear combination of all quadratic divisors orthogonal to $2$-roots with multiplicity one. An even lattice is called \textit{complete $2$-reflective} if it has a complete $2$-reflective modular form. 

\begin{lemma}[Lemma 4.1 in \cite{Wan21}]\label{lem:complete}
Let $M=U\oplus U(m)\oplus L$. If $M$ is complete $2$-reflective then any even overlattice of $M$ is also complete $2$-reflective.    
\end{lemma}

\begin{lemma}\label{lem:non-complete}
Let $M=2U\oplus L$ be a $2$-reflective lattice. If $M$ is not complete $2$-reflective, then there exists an even lattice $K$ such that $M\cong A_1\oplus K$.   
\end{lemma}
\begin{proof}
By assumptions, there exists a $2$-root $v$ of $M$ with $(v,M)=2\ZZ$, because the set of $2$-roots $u\in M$ with $(u,M)=\ZZ$ is transitive under the action of $\Orth^+(M)$ (see \cite[Proposition 3.3]{GHS09}). We conclude from \cite[Lemma 7.5]{GHS13} that $M = \ZZ v \oplus M_v$, where $M_v$ is the orthogonal complement of $v$ in $M$.  We then prove the lemma.  
\end{proof}

\section{A proof of Theorem \ref{MTH0}}\label{sec:Ma}
Ma \cite{Ma17} proved that the set of $2$-reflective lattices of signature $(n,2)$ with $n\geq 7$ is finite. We improve Ma's result by a new method. 

\begin{theorem}
There are only finitely many $2$-reflective lattices of signature $(n,2)$ with $n\geq 5$.     
\end{theorem}
\begin{proof}
We first prove the theorem for $n\geq 7$, which reproves Ma's result.
Let $M$ be a $2$-reflective lattice of signature $(n,2)$ with $n\geq 7$. By \cite[Lemma 4.8]{Ma17} there exists an even overlattice $M_1$ of $M$ with length $l(M_1)\leq 4$ and exponent $e(M_1)=e(M)$ or $e(M)/2$. By Lemma \ref{lem:U} we can write $M_1=2U\oplus L$. Lemma \ref{Lem:reductionMa} yields that $M_1$ is $2$-reflective. Applying Part (b) of the Gritsenko--Nikulin conjecture (proved by Looijenga \cite{Loo03}) or Borcherds' result \cite[Theorem 12.1]{Bor98} to $2U\oplus L$, we find that $U\oplus L$ is a $2$-reflective hyperbolic lattice. Nikulin and Vinverg have proved that there are only finitely many $2$-reflective hyperbolic lattices. Therefore, both the exponents $e(M)$ and $e(M_1)$ are bounded from above. We then prove the desired result. 

We then consider the remaining cases. Let $M$ be a $2$-reflective lattice of signature $(n,2)$ with $n=5$ or $6$. According to \cite[Lemma 4.8]{Ma17}, there exists an even overlattice $M_1$ of $M$ such that $e(M_1)=e(M)$ or $e(M)/2$, $l(A_{M_1})_2\leq 4$ and $l(A_{M_1})_p \leq 3$ for any odd prime $p$. 

If there is a $2$-reflective modular form on $\Orth^+(M_1)$ with simple zeros, then we conclude from \cite[Corollary 1.10]{Ma18} that the number of such $M_1$ is finite up to isomorphism.  Therefore, the exponent $e(M)$ is bounded from above. We then prove the finiteness of $M$. 

Suppose that there is no $2$-reflective modular form on $\Orth^+(M_1)$ with simple zeros. We claim that $M_1$ has a $2$-reflective modular form $F$ which vanishes on some quadratic divisor $v^\perp$, where $v\in M_1$ with $(v,v)=2$ and $(v,M_1)=2 \ZZ$. Otherwise, there would be a modular form on $\Orth^+(M_1)$ whose zero divisor is a linear combination of quadratic divisors $l^\perp$ with some fixed multiplicity $m$, where $l$ takes over $2$-roots of $M_1$ with $(l,M_1)=\ZZ$, because the set of these $l$ is transitive under $\Orth^+(M_1)$. Since $M_1$ splits $U$, by \cite[Corollary 1.3]{Bru14} the modular form $F$ can be constructed as a Borcherds product on some sublattice of $M_1$. Therefore, there exists a modular form $F_1$ with simple zeros such that $F=F_1^m$. This contradicts the assumption. 

The existence of $v$ yields a decomposition $M_1=A_1\oplus K$ for some $K$ with $l(A_K)_p\leq 3$ for any prime $p$. Therefore, we can write $K=U\oplus T$ and thus $M_1=U\oplus T\oplus A_1$. By Part (b) of the Gritsenko--Nikulin conjecture, the hyperbolic lattice $T\oplus A_1$ is $2$-reflective. This implies the finiteness of $M_1$. We then finish the proof.
\end{proof}

\section{A proof of Theorem \ref{MTH}}\label{sec:proof}
In this section we present a proof of Theorem \ref{MTH}. The proof is divided into six cases. 

\begin{theorem}\label{th:n>=13}
The lattices $2U\oplus 3E_8$, $2U\oplus 2E_8\oplus A_1$ and $2U\oplus 2E_8$ are the only $2$-reflective lattices of signature $(n,2)$ with $n\geq 13$.     
\end{theorem}
\begin{proof}
It was proved by Ma \cite[Proposition 3.1]{Ma17} that $2U\oplus 3E_8$ is the unique $2$-reflective lattice of signature $(n,2)$ with $n\geq 26$. 
We now assume that $13\leq n \leq 25$.  

Suppose that $M$ is a maximal even lattice of signature $(n,2)$ and it is $2$-reflective. The length of $A_M$ satisfies that $l(M)\leq 3$. By Nikulin's results \cite[Corollaries 1.10.2 and 1.13.3]{Nik80}, we can write $M=E_8\oplus K$ for some maximal even lattice $K$. By Lemma \ref{lem:maximal}, we can further write $K=2U\oplus L$. Thus we have a decomposition $M=2U\oplus E_8 \oplus L$ with $3\leq \rank(L)\leq 15$.

By \cite[Theorem 6.2]{Wan19}, the sublattice $R$ of $E_8\oplus L$ generated by $2$-roots has the full rank $n-2$. Moreover, we can decompose $R$ into irreducible root lattices of type $ADE$ as
$$
R=E_8 \oplus R_1 \oplus mA_1,
$$
where $m$ is some non-negative integer and $R_1$ is a direct sum of some irreducible root lattices not of type $A_1$ contained in $L$. All irreducible components of $R$ not of type $A_1$ are required to have the same Coxeter number. Therefore, if $R_1$ is not zero, then it has to be $E_8$, because $\rank(R_1)\leq 15$. By the last statement of \cite[Theorem 6.2 (c)]{Wan19}, we have the expression 
$$
E_8\oplus L=2E_8\oplus (n-18)A_1 \quad \text{or} \quad E_8\oplus (n-10)A_1.
$$

In the former case, the assumption that $M$ is maximal forces that $n-18\leq 3$.  When $n=18$, $M=2U\oplus 2E_8$. When $n=19$, $M=2U\oplus 2E_8\oplus A_1$. When $n=20$, by Lemma \ref{lem:2-elementary} we have
$$
M=2U\oplus 2E_8\oplus 2A_1 \cong 2U\oplus E_8\oplus D_{10}.
$$
The second model of $M$ contradicts \cite[Theorem 6.2 (b)]{Wan19}, because $E_8$ and $D_{10}$ have distinct Coxeter numbers. When $n=21$, it follows from Lemma \ref{lem:pullback} that $M= 2U\oplus 2E_8\oplus 3A_1$ is not $2$-reflective.

In the latter case, the assumption that $M$ is maximal forces that $n-10\leq 3$. When $n=13$, Lemma \ref{lem:2-elementary} yields
$$
M=2U\oplus E_8\oplus 3A_1 \cong 2U\oplus E_7 \oplus D_4,
$$
which contradicts \cite[Theorem 6.2 (b)]{Wan19}, because $E_7$ and $D_{4}$ have distinct Coxeter numbers. 

We now consider the general case. Let $M$ be a $2$-reflective lattice of signature $(n,2)$ with $13\leq n\leq 25$. It remains to show that $M$ has to be maximal. 

Suppose that $M$ is not maximal and $M_1$ is a maximal even overlattice of $M$. As a maximal $2$-reflective lattice, $M_1$ has to be $2U\oplus 2E_8\oplus A_1$ or $2U\oplus 2E_8$ by the discussions above. In particular, $n=19$ or $18$. For such $n$, we can adapt the above argument to show that $2U\oplus 2E_8\oplus A_1$ and $2U\oplus 2E_8$ are the only $2$-reflective lattices $M$ of signature $(n,2)$ and length $l(M)\leq 3$. 

We claim that the order of the group $M_1/M$ is not a prime, otherwise the order of $A_M$ would be $2p^2$ or $p^2$. Thus $l(M)\leq 3$, which forces that $M=M_1$, a contradiction. Therefore, there exists an even lattice $M_2$ such that $M< M_2 < M_1$ and $M_1/M_2$ is a nontrivial cyclic group. It follows that $l(M_2)\leq 3$ and thus $M_2=M_1$, a contradiction. We then finish the proof.  
\end{proof}

\begin{theorem}\label{th:n=12}
The lattice $2U\oplus E_8\oplus 2A_1$ is the unique $2$-reflective lattice of signature $(12,2)$.     
\end{theorem}
\begin{proof}
Let $M$ be a $2$-reflective lattice of signature $(12,2)$. By Lemma \ref{lem:exponent}, there exists an even overlattice $M_1=2U\oplus L$ of $M$ satisfying that $e(M)=e(M_1)$ and $l(M_1)\leq 5$.  By Lemma \ref{lem:2-roots}, there exists a class $T$ in the genus of $L$ which has $2$-roots. Since $M_1\cong 2U\oplus T$ is $2$-reflective and $T$ has $2$-roots, we conclude from \cite[Theorem 1.2]{Wan19} that $M_1\cong 2U\oplus E_8\oplus 2A_1$. Therefore, both $M$ and $M_1$ are $2$-elementary. Thus $M'/M\cong (\ZZ/2\ZZ)^a$ for some positive integer $a$. By Lemma \ref{lem:2-elementary}, $a\leq 14$ and it is an even integer. For each such $a$ there is a unique lattice $M$ up to isomorphism. To prove the theorem it suffices to show that none of the following lattices is $2$-reflective:
\begin{align*}
&2U(2)\oplus 10A_1 < U(2)\oplus U\oplus 10A_1 < 2U\oplus 10A_1< \\
<&2U\oplus D_4\oplus 6A_1 < 2U\oplus D_6\oplus 4A_1 < 2U\oplus D_8\oplus 2A_1.     
\end{align*}
This follows from \cite[Theorem 1.2]{Wan19} and Lemma \ref{Lem:reductionMa}. 
\end{proof}

\begin{theorem}\label{th:n=11}
There are exactly four $2$-reflective lattices of signature $(11,2)$:
$$
2U\oplus D_4\oplus 5A_1, \quad 2U\oplus 2D_4\oplus A_1, \quad 2U\oplus D_8\oplus A_1, \quad 2U\oplus E_8\oplus A_1. 
$$
\end{theorem}
\begin{proof}
The proof is similar to that of Theorem \ref{th:n=12}.
Let $M$ be a $2$-reflective lattice of signature $(11,2)$. By Lemma \ref{lem:exponent}, there exists an even overlattice $M_1$ of $M$ with $e(M)=e(M_1)$ and $l(M_1)\leq 5$. By a similar argument, we have a decomposition $M_1=2U\oplus L_1$ for some $L_1$ having $2$-roots, and then we show that $M_1$ is isomorphic to $2U\oplus E_8\oplus A_1$, or $2U\oplus D_8\oplus A_1$ or $2U\oplus 2D_4\oplus A_1$. Therefore, $M$ is $2$-elementary. We write $A_M \cong (\ZZ/2\ZZ)^a$. By Lemma \ref{lem:2-elementary}, $a\leq 13$ and it is an odd integer. For each such $a$ there is a unique lattice $M$ up to isomorphism. It remains to prove that none of the following lattices is $2$-reflective:
$$
2U(2)\oplus 9A_1 < U(2)\oplus U\oplus 9A_1 < 2U\oplus 9A_1 \cong 2U\oplus E_8(2)\oplus A_1.
$$
We derive from \cite[Theorem 6.2]{Wan19} that $2U\oplus E_8(2)\oplus A_1$ is not $2$-reflective, because $E_8(2)$ has no $2$-roots. We then finish the proof of the theorem. 
\end{proof}

\begin{theorem}\label{th:n=10}
There are exactly twelve $2$-reflective lattices of signature $(10,2)$: 
\begin{align*}
&2U\oplus E_8& &2U\oplus D_8& &2U\oplus 2D_4& &2U\oplus D_8'(2)& \\
&2U\oplus E_7\oplus A_1& &2U\oplus D_6\oplus 2A_1& &2U\oplus D_4\oplus 4A_1& &2U\oplus 8A_1& \\
&2U\oplus E_8(2)& &U\oplus U(2)\oplus E_8(2)& &U\oplus U(2)\oplus 8A_1& &2U(2)\oplus 8A_1.&
\end{align*}
\end{theorem}
\begin{proof}
Let $M$ be a $2$-reflective lattice of signature $(10,2)$. By Lemmas \ref{lem:exponent} and \ref{lem:2-roots}, there exists an even overlattice $M_1=2U\oplus L$ of $M$ satisfying that $e(M)=e(M_1)$, $l(M_1)\leq 5$ and $L$ has $2$-roots.  By \cite[Theorem 1.2]{Wan19}, we find that $M_1$ is isomorphic to $2U\oplus E_8$, or $2U\oplus D_8$ or $2U\oplus 2D_4$, or $2U\oplus E_7\oplus A_1$, or $2U\oplus D_6\oplus 2A_1$. This implies that both $M$ and $M_1$ are $2$-elementary. We write $A_M\cong (\ZZ/2\ZZ)^a$. By Lemma \ref{lem:2-elementary}, $a\leq 12$ and it is an even integer. When $a=0$, $M=2U\oplus E_8$. For any even $a\geq 2$ there are exactly two lattices $M$ up to isomorphism: one with level $2$ and the other with level $4$. Since $2U(2)\oplus E_8(2)$ has no $2$-roots, it is not $2$-reflective. 
\end{proof}

The (unique) $2$-reflective modular form on $U\oplus U(2)\oplus E_8(2)$ was first constructed by Borcherds \cite{Bor96} in the study of the moduli space of Enriques surfaces. Borcherds also showed that this form defines the denominator of the fake monster Lie superalgebra (see \cite{Bor92}).  The $2$-reflective modular forms on lattices $2U(2)\oplus mA_1$ for $1\leq m\leq 8$ were constructed by Gritsenko--Nikulin \cite[Section 6.2]{GN18}. These forms are identical to some reflective modular forms of weight $12-m$ on $2U\oplus D_m$. 

The last two cases (i.e. $n=8, 9$) are more subtle because there are $2$-reflective lattices which are not $2$-elementary and we cannot use Lemma \ref{lem:exponent} in a direct way. 

\begin{theorem}\label{th:n=9}
There are exactly nine $2$-reflective lattices of signature $(9,2)$:  
\begin{align*}
&2U\oplus D_7& &2U\oplus A_7& &2U\oplus E_7& &2U\oplus E_6\oplus A_1& 
&2U\oplus D_6\oplus A_1& \\ &2U\oplus D_4\oplus 3A_1& &2U\oplus 7A_1& &U\oplus U(2)\oplus 7A_1& 
&2U(2)\oplus 7A_1.&
\end{align*}
\end{theorem}
\begin{proof}
Let $M$ be a $2$-reflective lattice of signature $(9,2)$. We fix a maximal even overlattice $M_0$ of $M$. Combining Lemmas \ref{lem:maximal} and \ref{lem:2-roots}, we have a decomposition $M_0=2U\oplus L_0$ such that $L_0$ has $2$-roots. Since $M_0=2U\oplus L_0$ is $2$-reflective and $L_0$ has $2$-roots, we conclude from \cite[Theorem 1.2]{Wan19} that $M_0$ is isomorphic to $2U\oplus E_6\oplus A_1$, or $2U\oplus E_7$ or $2U\oplus D_7$. Notice that $M<M_0<M_0'<M'$. There exist positive integers $t$ and $a_j$ for $1\leq j\leq t$ such that
$$
M' / M_0' \cong (\ZZ / a_1\ZZ) \oplus \cdots \oplus (\ZZ/a_t\ZZ).
$$
For any $a_s$ there exists an even overlattice $M_1$ of $M$ such that $M<M_1<M_0<M_0'<M_1'<M'$ and $M_1'/M_0' \cong \ZZ / a_s\ZZ$ (and thus $M_0/M_1\cong \ZZ / a_s\ZZ$).  We next discuss by cases.

\vspace{2mm}

\textbf{(I)} $M_0=2U\oplus E_6\oplus A_1$. We claim that $M=M_0$. 

Suppose that there are some $a_s>1$. Then $\det(M_1)=6a_s^2$ and $l(M_1)\leq 3$. By Lemma \ref{lem:2-roots}, there exists an even positive definite lattice $L_1$ with $2$-roots such that $M_1=2U\oplus L_1$. Thus $M_1$ lies in the table of \cite[Theorem 1.2]{Wan19} as a $2$-reflective lattice, which leads to a contradiction by comparing determinants of lattices. Therefore, every $a_j$ is $1$ and then $M= M_0= 2U\oplus E_6\oplus A_1$. 

\vspace{2mm}

\textbf{(II)} $M_0=2U\oplus D_7$. We claim that $M=M_0$. 

Suppose that there are some $a_s>1$. Then $M_1$ has determinant $4a_s^2$, length $l(M_1)\leq 3$ and exponent $e(M_1)\geq 4$. Similarly to the previous case, $M_1$ is a $2$-reflective lattice in the table of \cite[Theorem 1.2]{Wan19}, which leads to a contradiction by comparing determinants and exponents of lattices. 

\vspace{2mm}

\textbf{(III)} $M_0=2U\oplus E_7$. We claim that either $M=2U\oplus A_7$ or $M$ is $2$-elementary. 

A similar argument shows that every $a_j$ is either $1$ or $2$. Therefore, there exists a non-negative integer $a$ such that
$$
M'/M_0' \cong (\ZZ/2\ZZ)^a.
$$
A subgroup $G$ of $M'/M_0'$ of order $d$ corresponds to an even lattice $M_G$ of determinant $2d^2$ satisfying that $M<M_G<M_0$ and $M_0 / M_G \cong G$. More precisely, 
$$
M_G=\{ x\in M_0 : (x,y)\in \ZZ, \; y\in G+M_0'\}. 
$$

\textbf{(1)} When $a=1$, $\det(M)=2^3$, $l(M)\leq 3$ and thus we can write $M=2U\oplus L$ such that $L$ has $2$-roots. By \cite[Theorem 1.2 (c)]{Wan19}, $M$ is isomorphic to $2U\oplus A_7$ or $2U\oplus D_6\oplus A_1$. 

\textbf{(2)} We now consider the case $a\geq 2$. Let $G=\ZZ/2\ZZ \times \ZZ/2\ZZ$ be a subgroup of $M'/M_0'$. Similarly to the case $a=1$, we find that the lattice $M_1$ corresponding to a subgroup $\ZZ/2\ZZ$ of $G$ is $2U\oplus A_7$ or $2U\oplus D_6\oplus A_1$. Suppose that $M_1=2U\oplus A_7$. Then we have that $M<M_G<M_1$, $\det(M_G)=2^5$ and $l(M_G)\leq 3$. It follows that the $2$-reflective lattice $M_G$ has a decomposition $2U\oplus L_G$ such that $L_G$ has $2$-roots, which yields that $M_G$ lies in the table of \cite[Theorem 1.2 (c)]{Wan19}. This leads to a contradiction by considering the determinant and the length. Therefore, $M_1=2U\oplus D_6\oplus A_1$. We see from \cite[Theorem 6.2 (c)]{Wan19} that $M_1$ is not complete $2$-reflective, that is, every $2$-reflective modular form on $M_1$ either has a quadratic divisor with multiplicity larger than $1$ or does not vanish on some quadratic divisor orthogonal to a $2$-root of $M_1$.  

By Lemma \ref{lem:exponent}, there exists an even overlattice $M_2=2U\oplus L_2$ of $M$ satisfying that $e(M_2)=e(M)$ and $l(M_2)\leq 5$. We choose the above $M_0$ as a maximal even overlattice of $M_2$. 

If $l(M_2) \neq 1$, i.e. $M_2\neq 2U\oplus E_7$, then we can choose $M_1$ such that $M_2< M_1=2U\oplus D_6\oplus A_1$. By Lemma \ref{lem:complete}, the $2$-reflective lattice $M_2$ is not complete $2$-reflective. According to Lemma \ref{lem:non-complete}, we can write $M_2=A_1\oplus K$. Since $\det(M)=2^{2a+1}$, we have $l(M_2)=l(A_1)+l(K)$, so $l(K)\leq 4$. Therefore, by Lemma \ref{lem:U} we can write $K=2U\oplus T$. Since $M_2=2U\oplus T\oplus A_1$ is $2$-reflective, it lies in the table of \cite[Theorem 1.2 (c)]{Wan19}. We then conclude that both $M$ and $M_2$ are $2$-elementary. 

We complete the proof by the classification of $2$-elementary lattices. 
\end{proof}

\begin{theorem}
There are exactly thirteen $2$-reflective lattices of signature $(8,2)$:  
\begin{align*}
&2U\oplus D_6& &2U\oplus A_6& &2U\oplus 2A_3& &2U\oplus 3A_2& &2U\oplus E_6& \\
&2U\oplus D_5\oplus A_1& &2U\oplus A_5\oplus A_1& &2U\oplus D_4\oplus 2A_1& &2U\oplus 6A_1& &2U\oplus E_6'(3)& \\
&U\oplus U(3)\oplus E_6'(3)& &U\oplus U(2)\oplus 6A_1& &2U(2)\oplus 6A_1&
\end{align*}
\end{theorem}
\begin{proof}
Let $M$ be a $2$-reflective lattice of signature $(8,2)$. We fix $M_0$ as a maximal even overlattice of $M$. Since $l(M_0)\leq 3$, we can represent $M_0=2U\oplus L_0$. By Lemma \ref{lem:2-roots}, we can assume that $L_0$ has $2$-roots. Since $M_0=2U\oplus L_0$ is $2$-reflective and $L_0$ has $2$-roots, we know from \cite[Theorem 1.2 (c)]{Wan19} that $M_0$ is isomorphic to $2U\oplus D_6$, or $2U\oplus A_6$, or $2U\oplus E_6$, or $2U\oplus D_5\oplus A_1$. Note that $M<M_0<M_0'<M'$. There exist positive integers $t$ and $a_j$ for $1\leq j\leq t$ such that
$$
M' / M_0' = (\ZZ / a_1\ZZ) \oplus \cdots \oplus (\ZZ/a_t\ZZ).
$$
For any $a_s$ there exists an even lattice $M_1$ such that $M<M_1<M_0$ and $M_1'/M_0' \cong \ZZ / a_s\ZZ$.  We next discuss by cases.    

\vspace{2mm}

\textbf{(I)} $M_0 = 2U\oplus A_6$. We claim that $M=M_0$. 

The above $M_1$ has determinant $7a_s^2$ and length $l(M_1)\leq 3$. By Lemmas \ref{lem:U} and \ref{lem:2-roots}, we have a decomposition $M_1=2U\oplus L_1$ such that $L_1$ has $2$-roots. Therefore, the $2$-reflective lattice $M_1$ lies in the table of \cite[Theorem 1.2 (c)]{Wan19}. We then find that $a_s$ has to be $1$. 

\vspace{2mm}

\textbf{(II)} $M_0 = 2U\oplus D_5\oplus A_1$. We claim that $M=M_0$. 

The above $M_1$ has determinant $2^3a_s^2$ and length $l(M_1)\leq 4$. We notice that $M_0$ is not complete $2$-reflective (see \cite[Theorem 6.2 (c)]{Wan19}). By Lemma \ref{lem:U}, $M_1$ splits $2U$. Thus Lemma \ref{lem:complete} yields that $M_1$ is not complete $2$-reflective. It follows from Lemma \ref{lem:non-complete} that we has a decomposition $M_1=A_1\oplus K$ with $l(K)\leq 4$. Therefore, we can write $K=2U\oplus T$ and then $M_1=2U\oplus A_1\oplus T$ by Lemma \ref{lem:U}. Thus the $2$-reflective lattice $M_1$ lies in the table of \cite[Theorem 1.2 (c)]{Wan19}. We then see that $a_s=1$. 

\vspace{2mm}

\textbf{(III)} $M_0=2U\oplus E_6$. We claim that $M=2U\oplus A_5\oplus A_1$ or $M$ has level $3$. 

\textbf{(1)} Suppose that there are some $a_s=2$. We show that $M=2U\oplus A_5\oplus A_1$. 

A subgroup $\ZZ/2\ZZ$ of $M'/M_0'$ induces an even lattice $M_1$ with $\det(M_1)=12$ and $l(M_1)\leq 2$. Therefore, by Lemmas \ref{lem:U} and \ref{lem:2-roots} the $2$-reflective lattice $M_1$ has an expression $M_1=2U\oplus L_1$ such that $L_1$ has $2$-roots. \cite[Theorem 1.2 (c)]{Wan19} then yields that $M_1= 2U\oplus A_5\oplus A_1$. If $M\neq M_1$ then there exists an even lattice $M_2$ satisfying that $M<M_2<M_1$ and $l(M_2)\leq 4$. Since $M_1$ is not complete $2$-reflective, by Lemma \ref{lem:complete} $M_2$ is not complete $2$-reflective, so we can write $M_2=A_1\oplus K$ with $l(K)\leq 4$ by Lemma \ref{lem:non-complete}. Therefore, we can represent $M_2=2U\oplus A_1\oplus T$ by Lemma \ref{lem:U}. By \cite[Theorem 1.2 (c)]{Wan19}, such a $2$-reflective lattice $M_2$ does not exist, leading to a contradiction. Therefore, $M=M_1=2U\oplus A_5\oplus A_1$. 

\textbf{(2)} Suppose that there is no $a_j = 2$. If there is $a_s>3$, then $\ZZ/a_s\ZZ$ induces a lattice $M_1$ with $\det(M_1)=3a_s^2$ and $l(M_1)\leq 3$. Therefore, we can write $M_1=2U\oplus L_1$ such that $L_1$ has $2$-roots. Clearly, such $2$-reflective lattice $M_1$ does not exist by \cite[Theorem 1.2 (c)]{Wan19}, a contradiction. Thus we can assume that
$$
M'/M_0' \cong (\ZZ / 3\ZZ)^t. 
$$
We next show that $M$ has level $3$. 

We denote the generators of $M'/M_0'$ by $v_i$ for $1\leq i\leq t$. 
Any subgroup $\latt{v_i}\cong \ZZ/3\ZZ$ induces an even lattice 
$$
M_i=\{ x\in M_0: (x,v_i) \in \ZZ \}
$$
with $\det(M_i)=3^3$ and $l(M_i)\leq 3$. 
Note that $M_i'$ is generated by $M_0'$ and $v_i$. By Lemmas \ref{lem:U} and \ref{lem:2-roots}, we can express $M_i=2U\oplus L_i$ such that $L_i$ has $2$-roots, and therefore $M_i$ lies in the table of \cite[Theorem 1.2 (c)]{Wan19}. We find that $M_i \cong 2U\oplus 3A_2$, so $3(v_i,v_i)\in 2\ZZ$ and $3v_i \in M_0$.

When $t>1$, for $i\neq j$ we define an even lattice
$$
M_{ij}=\{ x\in M_0: (x,v_i)\in \ZZ, \; (x,v_j)\in \ZZ\}
$$
with $\det(M_{ij})=3^5$.
Note that the dual lattice $M_{ij}'$ is generated by $M_0'$, $v_i$ and $v_j$. 

If $M_{ij}'/M_{ij}$ has elements of order $9$, then $l(M_{ij})\leq 4$. By Lemma \ref{lem:2-roots}, we can write $M_{ij}=2U\oplus L_{ij}$ for some $L_{ij}$ with $2$-roots. \cite[Theorem 1.2 (c)]{Wan19} implies that such a $2$-reflective lattice $M_{ij}$ does not exist. Therefore, each non-zero element of $M_{ij}'/M_{ij}$ has order $3$. 

We have thus proved that $M_{ij}'/M_{ij}=(\ZZ/3\ZZ)^5$, which implies that $M_{ij}$ has level $3$ and thus $M_{ij}\cong 2U\oplus E_6'(3)$. Thus $3(v_i,v_j)\in \ZZ$. It is easy to verify by definition that $M$ is of level $3$. 

Thus $M=U\oplus U(3)\oplus E_6'(3)$, $2U\oplus E_6'(3)$, $2U\oplus 3A_2$ or $2U\oplus E_6$. The lattice $2U(3)\oplus E_6'(3)$ has no $2$-roots, so it is not $2$-reflective. We remark that the complete $2$-reflective modular form on $U\oplus U(3)\oplus E_6'(3)$ is identical to the $6$-reflective modular form on $2U\oplus 3A_2$ by \cite[Lemma 2.2]{Wan22}.
 
\vspace{2mm}

\textbf{(IV)} $M_0=2U\oplus D_6$. We claim that $M$ is $2$-elementary or $M=2U\oplus 2A_3$. 

We can write
$$
M'/M_0' \cong (\ZZ / 2^{a_1}\ZZ)^{b_1} \oplus \cdots \oplus (\ZZ / 2^{a_t}\ZZ)^{b_t},
$$
otherwise there is an even lattice $M_1$ satisfying that $M<M_1<M_0$, $\det(M_1)=2^2 a^2$ for some odd integer $a$ and $l(M_1)\leq 2$. Thus we can write $M_1=2U\oplus L_1$ such that $L_1$ has $2$-roots. The $2$-reflective lattice $M_1$ contradicts \cite[Theorem 1.2 (c)]{Wan19}.

Assume that $M\neq M_0$. Let $v\in M'$ with $2v\in M_0'$ and $v\not\in M_0'$. We define 
$$
M_1=\{ x\in M_0: (x,v)\in \ZZ \}.
$$
Then $M_1'$ is generated by $M_0'$ and $v$. Note that $\det(M_1)=2^4$. We discuss by three cases.
\begin{enumerate}
    \item $M_1'/M_1=(\ZZ/2\ZZ)^4$. We show that $M$ is $2$-elementary.
    
    As a $2$-elementary lattice, $M_1=2U\oplus D_4\oplus 2A_1$. By replacing $M$ with an even overlattice of the same exponent (see Lemma \ref{lem:exponent}), we can assume that $l(M)\leq 5$. Then $M$ splits $2U$. Since $M_1$ is not complete $2$-reflective, we know from Lemma \ref{lem:complete} that $M$ is not complete $2$-reflective. Combining Lemma \ref{lem:non-complete} and Lemma \ref{lem:U} we have a
    decomposition $M=A_1\oplus K$ with $l(K)\leq 4$ and thus a decomposition $M=2U\oplus A_1\oplus T$. We then determine $M$ by \cite[Theorem 1.2 (c)]{Wan19} and find that it is $2$-elementary.  
    \item $M_1'/M_1=(\ZZ/4\ZZ)\oplus (\ZZ/2\ZZ)^2$.  
    Then $l(M_1)=3$ and thus we can express $M_1=2U\oplus L_1$ such that $L_1$ has $2$-roots. There is no such $2$-reflective lattice by \cite[Theorem 1.2 (c)]{Wan19}. 
    
    \item $M_1'/M_1=(\ZZ/4\ZZ)^2$. We show that $M=M_1=2U\oplus 2A_3$.
    
    In this case, $l(M_1)=2$ and thus $M_1$ is a $2$-reflective lattice in the table of \cite[Theorem 1.2 (c)]{Wan19}. It follows that $M_1=2U\oplus 2A_3$. Assume that $M\neq M_1$. We take a lattice $M_2$ satisfying that $M<M_2<M_1<M_1'<M_2'<M$ and $M_2'/M_1'=\ZZ/2\ZZ$. When $l(M_2)\leq 3$, we can express the $2$-reflective lattice $M_2$ as $2U\oplus L_2$ such that $L_2$ has $2$-roots. By \cite[Theorem 1.2 (c)]{Wan19}, such $M_2$ does not exist. 
    
    Therefore, $l(M_2)>3$ and further $M_2'/M_2 \cong (\ZZ/4\ZZ)^2\oplus (\ZZ/2\ZZ)^2$. There are two cases:
    \begin{enumerate}
        \item $M_2 \cong U\oplus U(2)\oplus 2A_3$. We observe that $U\oplus U(2)\oplus 2A_3 \cong 2U\oplus L_2$ for some $L_2$ with $2$-roots. By \cite[Theorem 1.2 (c)]{Wan19}, such $M_2$ does not exist, a contradiction. 
        \item $M_2 \cong U\oplus A_1\oplus A_1(-1)\oplus 2A_3$. By Lemma \ref{lem:U}, we have $A_1(-1)\oplus 2A_3 \cong U\oplus T$ for some $T$. Therefore,
        $$
        U\oplus A_1\oplus A_1(-1)\oplus 2A_3 \cong 2U\oplus A_1\oplus T. 
        $$
        By \cite[Theorem 1.2 (c)]{Wan19}, such $M_2$ does not exist, a contradiction. 
    \end{enumerate}
\end{enumerate}
We finish the proof by the discussions above and the classification of $2$-elementary lattices. 
\end{proof}

At the end of this section, we give three corollaries of the main theorem. 

\begin{corollary}\label{cor1}
There are exactly $21$ complete $2$-reflective lattices of signature $(n,2)$ with $n\geq 8$ up to isomorphism. They are formulated as follows:
\begin{align*}
&2U\oplus 3E_8, \quad 2U\oplus 2E_8, \quad 2U\oplus E_8, \quad 2U\oplus E_8(2), \quad U\oplus U(2)\oplus E_8(2), \quad 2U\oplus D_8, \quad 2U\oplus 2D_4, \\  &2U\oplus D_8'(2), \quad 2U(2)\oplus 8A_1, \quad 2U\oplus D_7, \quad 2U\oplus A_7, \quad 2U\oplus E_7, \quad 2U(2)\oplus 7A_1, \quad 2U(2)\oplus 6A_1, \\  &2U\oplus D_6, \quad 2U\oplus A_6, \quad 2U\oplus 2A_3, \quad 2U\oplus 3A_2, \quad 2U\oplus E_6, \quad 2U\oplus E_6'(3), \quad U\oplus U(3)\oplus E_6'(3).
\end{align*}
\end{corollary}
\begin{proof}
It is a direct consequence of \cite[Theorem 6.9]{Wan19} and Theorem \ref{MTH}. The lattice $U\oplus U(2)\oplus mA_1$ is not complete $2$-reflective for $6\leq m \leq 8$, because its even overlattice $2U\oplus mA_1$ is not complete $2$-reflective (see Lemma \ref{lem:complete}). 
\end{proof}

The weights of complete $2$-reflective modular forms on $14$ of the above $21$ lattices are formulated in \cite[Table 2]{Wan19}. The complete $2$-reflective modular form has weight $12$ on $2U\oplus E_8(2)$ and $2U\oplus E_6'(3)$, weight $12-m$ on $2U(2)\oplus mA_1$ for $m=6,7,8$, weight $4$ on $U\oplus U(2)\oplus E_8(2)$ and weight $3$ on $U\oplus U(3)\oplus E_6'(3)$.

\begin{corollary}
Let $L$ be a primitive sublattice of the Leech lattice satisfying the $\mathrm{Norm}_2$ condition, that is, for any $\gamma \in L'/L$ there exists $v\in L+\gamma$ such that $(v,v)\leq 2$. 
If the rank of $L$ is greater than $5$, then $L$ is isomorphic to $E_8(2)$, $E_6'(3)$ or the Leech lattice. 
\end{corollary}
\begin{proof}
Let $\Lambda$ denote the Leech lattice.
By \cite[Section 5.1]{Wan19}, the pullback of the Borcherds form on $2U\oplus \Lambda$ defines a complete $2$-reflective modular form of weight $12$ on $2U\oplus L$.  Note that $L$ has no $2$-roots. The result then follows from the above corollary. 
\end{proof}

\begin{corollary}
Let $M$ be an even lattice of signature $(n,2)$ with $n\geq 8$. If the ring of integral-weight modular forms for the discriminant kernel
$$
\widetilde{\Orth}^+(M) = \{ g \in \Orth^+(M) : g(x) - x \in M, \; \text{for all $x\in M'$} \}
$$
is freely generated by $n+1$ forms, then $M=2U\oplus L$ for $L=E_8$, $D_8$, $D_7$, $A_7$, $E_7$, $D_6$, $A_6$ or $E_6$. 
\end{corollary}
\begin{proof}
Suppose the ring of modular forms for $\widetilde{\Orth}^+(M)$ is freely generated by forms $F_i$ of weights $k_i$ for $1\leq i\leq n+1$. By \cite[Theorem 3.5]{Wan21}, the Jacobian of these $F_i$ is a complete $2$-reflective modular form of weight $n+\sum_{i=1}^{n+1}k_i$.  We then complete the proof by Corollary \ref{cor1} and \cite[Theorem 4.4]{Wan21}. 
\end{proof} 

\bigskip

\noindent
\textbf{Acknowledgements} 
The author is supported by the Institute for Basic Science (IBS-R003-D1). The author thanks Shouhei Ma for valuable discussions.

\bibliographystyle{plainnat}
\bibliofont
\bibliography{refs}

\end{document}